\documentclass [12pt,a4paper,reqno]{amsart}
\usepackage{amssymb}
\usepackage{verbatim}
\usepackage{pgf}
\usepackage{needspace}
\usepackage{longtable}
\usepackage[a4paper,margin=1in,includehead,includefoot]{geometry}
\usepackage{nccmath}

\usepackage{amsthm}
\usepackage{amsmath}
\usepackage{stmaryrd}
\usepackage{stackrel}
\usepackage{setspace}
\PassOptionsToPackage{normalem}{ulem}
\usepackage{ulem}

\makeatletter

\numberwithin{figure}{section}
\numberwithin{equation}{section}
\numberwithin{table}{section}

  \input amssymb.sty
\usepackage{etoolbox}
\patchcmd{\thebibliography}{\section*}{\section}{}{}

\usepackage{algpseudocode} \usepackage{algorithm}
\usepackage{algpseudocode,algorithm}
\usepackage{lipsum}
\usepackage{caption}
\usepackage{graphics}

\usepackage{amsmath}
\usepackage{amssymb}
\usepackage[all]{xy}

\usepackage{epsfig}\usepackage{youngtab}

\newcommand{\ef}{\end{equation}}
\chardef\bslash=`\\ 

\newdir{:=}{{}}
\newcommand*\colvec[3][]{
    \begin{pmatrix}\ifx\relax#1\relax\else#1\\\fi#2\\#3\end{pmatrix}
}

\newtheorem*{thm*}{Theorem}

\newtheorem{lem}{Lemma}[section]
\newtheorem*{lem*}{Lemma}

\newtheorem*{corl*}{Corollary}

\newtheorem{prop}{Proposition}[section]
\newtheorem{prop*}{Proposition}

\newtheorem{claim}{Claim}

\theoremstyle{definition}
\newtheorem{defn}{Definition}[section]
\newtheorem{examp}{Example}

\newtheorem*{examp*}{Example}

\newtheorem*{cor*}{Corollary}
\newtheorem*{remark*}{Remark}
\newtheorem*{CC*}{Crossover Conjecture}
\newtheorem*{Note*}{Note}
\newtheorem*{defn*}{Definition}

 \theoremstyle{remark}

 \renewcommand{\sectionmark}[1]{}

\newcommand{\defect}{\operatorname{def}}

\renewcommand{\a}{\alpha}
 \newcommand{\inv}{\operatorname{inv}}
   \newcommand{\Inv}{\operatorname{Inv}}

\begin{document}

\title[Canonical Basis Computations]{Non-recursive Canonical Basis Computations  for Low Rank Kashiwara Crystals of Type $A$}

 \author{Ola Amara-Omari}
\address{Department of Mathematics\\
Bar-Ilan University, Ramat-Gan\\
Israel}
\email{olaomari77@hotmail.com}

 \author{Mary Schaps}
\address{Department of Mathematics\\
Bar-Ilan University, Ramat-Gan\\
Israel}
\email{mschaps@macs.biu.ac.il}

 \thanks{Partially supported by Ministry of Science, Technology and Space fellowship, at Bar-Ilan University, Ramat-Gan, Israel. Some of the results appear in the Ph.D. thesis of the first author.}
 
 
 \begin{abstract}
For symmetric Kashiwara crystals of  type $A$ and  rank $e=2$, and for the canonical basis elements that we call external, corresponding to weights on the outer skin of the Kashiwara crystal, we construct the canonical basis elements in a non-recursive manner. In particular, for a symmetric crystal with $\Lambda=a \Lambda_0+a \Lambda_1$, we give formulae for the canonical basis elements for all the $e$-regular multipartitions with defects either $k(a-k)$ or $k(a-k)+2a$, for $0 \leq k \leq a$.
\end{abstract}

\maketitle

\section{INTRODUCTION}

The highest weight representations of the enveloping algebra of an affine Lie algebra have been intensively studied, through the Kashiwara crystal $B(\Lambda$) \cite{K1},\cite{K2}  corresponding to the  chosen highest weight $\Lambda$. For Lie algebras of type $A$, there are three  different representation for the basis elements, through multipartitions, through Littelmann paths \cite{L} and through canonical bases.  In an earlier paper \cite{OS}, we studied the passage from multipartitions to Littelmann paths, and in this paper we consider the passage from multipartitions to canonical basis elements. In both cases, the use of computer algebra was critical in making calculations and formulating conjectures.

The Kashiwara crystal, long studied for its importance to the representation theory of the quantum enveloping algebra, is also important because of the categorification theory of Chuang and Rouquier, \cite{CR}. Under categorification, the canonical basis elements that we will study correspond to simple modules in blocks of cyclotomic Hecke algebras.

Our original conjectures were obtained by experimentation, using programs in Sagemath \cite{Sa} or are privately available to interested researchers. 
In generating the highest weight representation of a particular dominant integral weight, the particular object from which we start is the object CrystalOfLSPaths(CartanType), whose authors are Mark Shimozono and Anne Schilling. This program generates the Littelmann paths.

In the case of type $A$, there are two additional representations of the basis elements of the crystal, one by multipartitions and one by canonical basis elements.  Of these three representations, the multipartition is the most compact, the Littelmann path is second and the canonical basis is the most verbose.  We choose the Littelmann path model, which is available for all supported Cartan types, as our primary representation.

Our own program made use of an implementation by Travis Scrimshaw of an algorithm of Matthew Fayers \cite{Fa2}, extending the algorithm in  \cite{LLT},  for constructing the canonical bases and a variant of Kleshchev's algorithm \cite{Kl}  for constructing $e$-regular multipartitions recursively.  Our program, CanonicalBasisfromPaths(CartanType, HighestWeight) proceeds recursively degree by degree.  For each basis element $w$, we keep track of all the different paths which can be used to reach  $w$. Each basis element has a unique multipartition, and by using the signature method \cite {Kl}, we can calculate the new multipartition.  From the point of view of information theory, there is obviously a great deal of redundancy in the representation by canonical basis, for which the multipartition can be obtained immediately. It corresponds to the only Fock space basis element with coefficient $1$.

\section{DEFINITIONS AND NOTATION}

Let $\mathfrak g$ be the affine Lie algebra $A^{(1)}_{e-1}$, untwisted of type $A$, with a Dynkin diagram which is a piecewise linear circle. Let $C$ be the Cartan matrix, and $\delta$ the null root. Let $\Lambda$ be a dominant integral weight, let $V(\Lambda)$ be the highest weight module with that
highest weight, and let $P(\Lambda)$ be the set of weights of $V(\Lambda)$ \cite{Ka}. Let $Q$ be the $\mathbb Z$-lattice generated by the simple roots,
\[
\a_0,\dots,\a_{e-1}.
\]
Let  $Q_+$ be the subset of $Q$ in which all coefficients are non-negative.

The weight space $P$ of the affine Lie algebra has two different bases.  One is given by the fundamental weights and null root, $\Lambda_0,\dots, \Lambda_{e-1}, \delta$,
and one is given by $\Lambda_0, \a_0,\dots,\a_{e-1}$.  We will usually use the first basis for our weights.

The highest weight module $V(\Lambda)$ is integrable.  Every weight $\lambda$ in $P(\Lambda)$ has the form $\Lambda-\alpha$, for $\alpha \in Q_+$.  The vector
of nonnegative integers giving the coefficients of $\alpha$ is called the content of $\lambda$. We follow \cite{Kl} in defining the \textit{defect} by

\[
\defect(\lambda)=\frac{1}{2}((\Lambda \mid \Lambda)-  (\lambda \mid \lambda))=(\Lambda \mid \alpha)-\frac{1}{2}(\alpha \mid \alpha).
\]

 An alternative definition directly from the content of a multipartition $\mu$ is given in \cite{Fa}. where it is called the weight of the multipartition and denoted by $w(\mu)$. Since we are in a highest weight module, we always have $(\Lambda \mid \Lambda) \geq (\lambda \mid \lambda)$, so the defect is non-negative and is, in fact, an integer for the affine Lie algebras of type $A$ treated in this paper. The weights of defect $0$ are those lying in the Weyl group orbit of $\Lambda$ and play an important role.
Let $W$ denote the Weyl group, generated by reflections $s_0,\dots, s_{e-1}$. For the affine Lie algebras, the Weyl group is infinite, being the semidirect product of a finite Weyl group acting on an infinite abelian subgroup. \cite{Ka}

A partition $\lambda=(\lambda_1,\lambda_2,\dots,\lambda_t)$ is a sequence of integers with $\lambda_1 \geq \lambda_2 \geq \dots \geq \lambda_t$ of length $\ell(\lambda)=t$.  A multipartition  $\lambda=(\lambda^1,\lambda^2,\dots,\lambda^r)$ is a sequence of partitions.  The \textit{dominance order} on multipartitions is given by
$\mu \trianglerighteq \lambda$ if, for all  integers $k$ with $1 \leq k \leq r$ and $j \leq \ell(\mu^k)$,
\[
\sum_{\ell=1}^{k-1}\mid \mu^\ell \mid + \sum_{i=1}^j \mu^k_i \geq \sum_{\ell=1}^{k-1}\mid \lambda^\ell \mid + \sum_{i=1}^j \lambda^k_i.
\]

Our quantum enveloping algebra will be $\mathcal U= \mathcal U_v(\hat sl(e))$, where we are using Lusztig's $v$ in place of the more common quantum parameter $q$ because we are using balanced quantum integers $ [n]_v = v^{n-1}+v^{n-3}+\dots+v^{-(n-3)}+v^{-(n-1)}$.  The quantum factorial is $ [n]_v!= [n]_v \cdot  [n-1]_v \cdot \dots  [1]_v$. The underlying ring of the enveloping algebra is $\mathbb Q(v)$, and the generators are $e_i, f_i,h_i$  for $i \in I=\mathbb Z/\mathbb Ze$ and a central element $c$.

Choose a sequence  $s=(k_1,\dots,k_r)$ such that $\Lambda=\Lambda_{k_1}+\dots+\Lambda_{k_r}$.  In type $A$, the number $r$ of terms in the sum is called the level. The definition of level is more complicated for other types.

The Fock space $\mathcal F^s$  is a space over $\mathbb Q(v)$ with basis given by multipartitions consisting of $r$ partitions.
For a multipartition $\lambda$ with Young diagram $Y(\lambda)$, the node $(t,u)$ in partition $\ell$ is given residue
\[
k_\ell+u-t \mod e.
\]
An addable $i$-node  $\mathfrak n$ is a node of residue $i$  outside $Y(\lambda)$ such that if added it would give a multipartition, which we denote by  $\lambda^{\mathfrak n}$.
A removable $i$-node $\mathfrak m$ inside a multipartition $\mu$ is a node of residue $i$ at the end of a row or column which, if removed,  would give a multipartition, which we denote by  $\mu_{\mathfrak m}$.
The quantum enveloping algebra acts on the Fock space by determining actions for the elements of the Chevalley basis, as follows:
\begin{itemize}
\item  For an addable node, let us now  define  $N(\mathfrak n,i)=\#\{$   addable  $i$-nodes above~$\mathfrak n\}-\# \{$ removable $i$-nodes above $\mathfrak n \}$ and set
\[
f_i(\lambda)=\sum_{\mathfrak n}v^{N(\mathfrak n,i)}\lambda^{\mathfrak n}.
\]

\item  For a removable node, let us now define $M(\mathfrak m,i)=\#\{$ addable $i$-nodes below~$\mathfrak m\}-  \# \{$ removable $i$-nodes below $\mathfrak m\}$.
\[
e_i(\mu)=\sum_{\mathfrak m}v^{M(\mathfrak m,i)}\mu_{\mathfrak m}.
\]
\end{itemize}
The divided powers are $e_i^{(k)}$ and $f_i^{(k)}$ and they are given by dividing by the quantum factorials $[k]_v!$.
 We define $\mathcal F^s_{\mathcal A}$ to be the subalgebra of $\mathcal F^s$ generated by the divided powers from the highest weight vector  over $\mathcal A$, where
coefficients lie in the algebra $\mathcal A$ of Laurent polynomials in $v$ with integral coefficients.
  In addition, there is an involution of the quantum enveloping algebra called the bar-involution which fixes $e_i$, $f_i$ and $h_i$, but interchanges $v$ and $v^{-1}$. For each $e$-regular partition $\mu$, there is an element $G(\mu)$ of the Fock space $\mathcal F^s_{\mathcal A}$ that is invariant under the bar involution. and these are called the canonical basis elements. The action of the Chevalley basis elements $e_i$ and $f_i$ on these canonical basis elements is induced from their action on the basis elements of the Fock space.

The $e$-regular  multipartitions in the Kashiwara crystal $B(\Lambda)$ are generally calculated recursively.  For a multipartition $\lambda$ and any residue $i$  there is at most one node of  residue $i$ which can be added, and it is chosen by the signature method \cite{Kl} as follows:  for any given residue $i$, we consider all the addable $i$-nodes to be those nodes at the ends of rows whose addition would still leave us with a partition, denoted by a ``+'', and the $i$ nodes which are the last node in a row will be called removable if they can be removed and still leave a partition, denoted by a ``-''. We then write from left to right  all the pluses and minuses from the bottom to the top, remove any cases of ``-+'', and call the remaining sequence of plus and minus signs the \textit{signature}.  The first removable $i$-node from the left is called $i$-good, and the first addable $i$-node from the right is called $i$-cogood.  The operation of $e_i$ is then the removal of the $i$ good node if it exists  and otherwise gives $0$, while the operation of $f_i$ is the addition of the $i$-cogood node if it exists and otherwise  gives $0$. This procedure, starting with the highest weight vector and acting by various $f_i$ will produce all the elements of the crystal, and will give multipartitions in which every partition is $e$-regular, i.e., does not have $e$ identical rows.  There is an analogous construction, usually preferred by Brundan and Kleshchev \cite{Kl}, which produces multipartitions in which every partition is $e$-restricted, which means that there are no $e$ consecutive columns which are equal.

In \cite{AS}, we introduced a graph $\hat P(\Lambda)$  with vertices $P(\Lambda)$ and edges between two weights when there is an edge in the Kashwara crystal between two basis elements with those weights. We called it the \textit{block-reduced crystal graph}, various properties of which are described in \cite{BFS}. It is easily calculated in polynomial time, because the lengths of the $i$-strings are given by the positive entries in the hubs, the projection of the weight onto the fundamental weights.

\section{CANONICAL BASIS ELEMENTS}

For $e=2$, Mathas \cite{M1} completely determined the $e$-regular multipartitions.  Ariki, Kleiman and Tsuchioka \cite{AKT} did the same for the $r=2$, using the Littelmann path model.  Less in known about the possibility of non-recursive construction of the canonical basis elements. In \cite{OS}, for $e=2$, we managed to show that for multipartitions we called external, which lie, as it were, on the crust of the crystal, we can move directly back and forth between the multipartitions and the Littelmann paths.  We would like to find some such procedure for external canonical basis elements.

We introduce some notation which will allow us to describe families of multipartitions:
\begin{itemize}
\item $(n)$ is a row of of length n,
\item $T_n$ is triangular with $n$ rows if $n>0$ or $\emptyset$ otherwise,
\item $\lambda \vee \mu$ is the partition obtained by taking the rows in $\lambda$ followed by the rows in $\mu$, presuming that this a well-formed partition.
\end{itemize}

As mentioned before, we will follow Mathas in \cite{M1} in assuming that all partitions with the same corner residue will lie in an interval in the multipartition, and we will also take them in increasing order, $\Lambda=a_0\Lambda_0+a_1\Lambda_1+\dots+a_{e-1}\Lambda_{e-1}$. In the case $e=2$, the  notation $\Lambda = a\Lambda_0+b\Lambda_1$ will indicate that $k_1,k_2,\dots,k_a=0$, while $k_{a+1},\dots,k_{a+b}=1$. For $e=2$, we also put a semicolon between the $0$-corner partitions and the $1$-corner partitions.
 If $i$ is a residue in the set $\{0,1\}$, we let $i'=1-i$ be the opposite residue.

In \cite{Fa}, Fayers describes two involutions on the multipartitions:
\begin{defn} If  $\lambda=(\lambda^1,\dots, \lambda^r)$ is a multipartition of rank $e$ and level $r$, then the \textit{conjugate} $\lambda'$ of $\lambda$ is given by  $\lambda'=({\lambda^r}',\dots, {\lambda^1}')$, where ${\lambda^i}'$ is the transposed partition of ${\lambda^i}$ , corresponding to reflection of the Young diagram in the main diagonal.
\end{defn}

\begin{defn} If  $\lambda=(\lambda^1,\dots, \lambda^r)$ is a multipartion of rank $e$ and level $r$ for $\Lambda=\Lambda_{k_1}+\dots\Lambda_{k_r}$, then the \textit{diamond} $\lambda^\diamond$ of $\lambda$ is a multipartition in the crystal for  $\hat\Lambda=\Lambda_{-k_r}+\dots\Lambda_{-k_1}$, whose path is obtained from a path giving $\lambda$ by replacing each residue by minus that residue.

\end{defn}

In Theorem 2.1 of \cite{Fa}, Fayers proves that if $w(\mu)$ is the defect of an $e$-regular multipartition $\mu$, then
\[
\hat d_{\lambda' \mu^\diamond}=v^{w(\mu)}d_{\lambda \mu}(v^{-1}).
\]

\noindent If $\mu$ is an $e$-regular multipartition and we denote by $G(\mu)$ the corresponding canonical basis element, then this theorem implies that if $w(\mu)$ is the defect of the block containing $\mu$, then there is a unique Fock space basis element in $G(\mu)$ with coefficient $v^{w(\mu)}$ and it is given by $(\mu^\diamond)'.$ This already gives us some information about low defects, without any other restrictions:

\begin{itemize}
\item \textit{Defect 0:} The multipartitions of defect $0$ are precisely those for which $\mu= (\mu^\diamond)'$. For $e=2$, we already know from \cite{M1} that the defect $0$ multipartitions for $\Lambda=a\Lambda_0+b\Lambda_1$ consist of $a$ triangular partitions of side $n$ and $b$ triangular partitions of side $n \pm 1$.  The diamond operation reverses the order, and the prime operation reverses the order back and transposes all the partitions, which is not noticed because they are triangular and thus invariant under transpose.
\item \textit{Defect 1:} The canonical basis is $G(\mu)=\mu+v(\mu^\diamond)'$.
\end{itemize}

One more word about  $(\mu^\diamond)'$, a result of the strong duality of the canonical basis element.  We chose $\mu$ to be $e$-regular.  Then $(\mu^\diamond)'$ is $e$-restricted, that is to say, it is the partition we would get from  the same path through the crystal if we were always calculating our signatures from the bottom to the top instead of from the top to the bottom as we do. We thus have four partitions exhibiting two sorts of duality.

\section{NON-RECURSIVE CONSTRUCTIONS}

From the results of \cite{BFS}, there is a fundamental region of $P(\Lambda)$ under the action of the normal translation subgroup of the Weyl group, and the defects which can occur in the basis graph $P(\Lambda)$ are all congruent to these defects modulo $r$, since subtracting the null root $\delta$ adds the level $r$ to the defect. In the case of rank $e=2$, since the null root $\delta = \alpha_0+\alpha_1$, the defects which can occur, modulo $r$,  all occur on the top row on the right and left, and are given by the following lemma. For each such weight there is a unique multipartition, and we will say that this multipartition is \textit{
derived from the top row}.

\begin{lem}\label{k}
In a  crystal with $\Lambda=a\Lambda_0 + b\Lambda_1$, the defect of $\lambda=\Lambda-k\alpha_0$ for $ 0 \leq k \leq a$ is $k(a-k)$  and of   $\lambda=\Lambda-\ell\alpha_1$ for  $0 \leq \ell \leq b$  is $\ell(b-\ell)$.
\end{lem}

\begin{proof}
We simply compute the defect explicitly for the case of $\alpha_0$, the other case being symmetric:
\[
\defect(\lambda)=(\Lambda|k\alpha_0)-\frac{1}{2}(k\alpha_0|k\alpha_0)
\]

\[\
ka-\frac{1}{2}k^2(\alpha_0|\alpha_0)=ka-k^2
\]
\end{proof}
\begin{defn} The \textit{shape} of a canonical basis element is the number of multipartitions, counting repetitions,  for each power of $v$ between $0$ and the defect.  A canonical basis element of defect $d$ whose shape is  $(1,1,1,\dots,1)$ with $d+1$ entries will be called \textit{svelte}.
\end{defn}

Theorem 2.1 in \cite{Fa}, giving the relationship between $\mu$ and $\mu^\diamond$, is sometimes awkward to use, particularly in computer algebra computations, because it involves constructing two distinct crystals and comparing them.
However, in the rank $2$ symmetric case,  $\Lambda=a\Lambda_0+a\Lambda_1$, we can do everything within the confines of a single crystal, and thus gain considerable information about the coefficients in the canonical basis.  Our question is this:  To what extent can we determine the canonical basis element from the multipartition and the block-reduced crystal $\hat P(\Lambda)$  without resorting to recursive calculations?

Our strategy for giving a non-recursive construction of canonical basis elements is to find a uniform notation for all the multipartitions occuring in the canonical basis elements.  We take a residue-collected  path $(0^{k^1},1^{k^2},0^{k^3},\dots)$ or  $(1^{k^1},0^{k^2},1^{k^3},\dots)$ in $P(\Lambda)$ starting with residue $i$, and let $(d_1,d_2,\dots) $ be the defects of the weights at the ends of each fixed-residue string. 
\begin{defn}\label{stab} We say that the path is \textit{stabilizing at $t$} if the defects rise to a fixed defect $d$ in place $t$ and afterwards are all $d$, so that in fact from $t$ forward, the actions are actions of Weyl group elements reflecting the strings.
\end{defn}
 For such a stabilizing path, let $c_1=a$ be the number of $i$-addable nodes in the highest weight vector $u_\Lambda$, and let $S_1$ be a characteristic sequence of length $c_1$ choosing $k^1$ nodes addable nodes out of $c_1$. Given such a choice, we have a new number $c_2=2k^1+a$ of $i'$- addable nodes for the second residue, and make a new choice of $S_2$ among them.    In general, we  let the $S_\ell$ be characteristic functions of subsets of size $k^\ell$ in  $[1,2,\dots,c_\ell]$, where $c_\ell$ may depend on the previous choice functions $S_j$, $j<\ell$. Let $\mathcal S(c,k)$ be the set of all choice sequences of length $c$ with $k$ $1$-entries.

The method is simplest to apply when the number $c_i$ of addable nodes  is independent of the various choices $S_1, S_2, \dots, S_{i-1}$ made previously, but we will show in Example \ref{dif} below that this is not always true.
What affects the value of $c_\ell$ the most is the distribution of $S_\ell$ between $0$- and $1$-corner nodes, so we let $S^0_\ell$ be the part of the sequence $S_\ell$  lying in $0$-corner nodes, and let $S^1_\ell$ be the subsequence  of $S_\ell$ lying in the $1$-corner nodes, so that we obtain $S_\ell$ by concatenation, $S_\ell=S^0_\ell \vee S^1_\ell$. At each stage, we let $\tilde c_\ell$ be the number of addable nodes in the $e$-regular multipartition, and let $\tilde S_\ell$ be the choice sequence of the $e$-regular multipartition, which will usually have all the $1$-entries at the beginning, unless there are problems of $-+$ in the signature.

We must also define  new families of partitions depending on $n$ for $n \geq 1$.  We let
\[
U^n_1=(n+1)\vee T_{n-2};U^n_2 =T_{n-1} \vee (1^2)
\]
 be families built from  $(2)$ and $(1^2)$ by alternately adding all $i$-addable nodes. Note that $U^n_2$ is the transpose of $U^n_1$.

What do we mean by ``non-recursive''?  We will rely heavily on the block-reduced crystal graph $\hat P(\Lambda)$, \cite{AS} which is most easily computed recursively, but which does have a non-recursive construction, given in \cite{BFS}.  We will need a choice tree out to a preliminary weight space of the desired defect, controlling the length of the strings of addable nodes, after which we will use Weyl group generators.   For the strongly residue-homogeneous multipartitions  treated in \cite{OS}, it may be possible to determine the structure of the canonical basis element directly from the segment structure of the multipartition, and then show that the shape remains stable under all further actions of the Weyl group. What we are hoping to avoid is the current situation that in order to go down an $i$-string from one basis element of defect $d$ to another, we  have to construct the mammoth canonical basis elements in between, which blows up in an exponential manner.

\begin{defn}\label{Inv}
For a sequence $S$ chosen from a two element ordered set $\{0,1\}$, the number of inversions, $\Inv(S)$, is the  sum of the number of elements $0$ appearing before each element $1$.
\end{defn}

The following conjecture summarizes the results of numerous computer calculations of canonical basis elements. Recall that a path stabilizing at $t$ was defined in Def. \ref{stab}. 

\noindent \textbf{Conjecture:} Let $e=2$ and $\Lambda=a \Lambda_0+a\Lambda_1$. Let $\mu$ be an $e$-regular multipartition reached by a path $p$ stabilizing at $t$  of length $w$, $p=(  i^{k^1},(i')^{k^2},\dots)$.   Let $t'$ be the first index greater than or equal to $t$ for which we get an external
weight space, if such $t'$  exist, and otherwise let it be $w$. Set $n=w-t$. Then there is a number $m$ with $t \leq m  \leq t'$ such that
\[
G(\mu)=\sum_{S_1 \in \mathcal S(c_1,k^1)} \dots \sum_{S_m \in \mathcal S(c_m,k^m)}v^{\Inv(S_1)+\dots+\Inv(S_m)} \pi^{n} (S_1,S_2,\dots,S_m)
\]
where the $\pi^{n}$ are multipartitions determined entirely by the choices of addable nodes given by the $S_\ell$. If $w>t'$, all the canonical basis element from $t'$ up have the same shape.

\bigskip

In this paper, we will verify the conjecture for a number of cases  for which $t'$ and $t$ are small. If $S_\ell$ has a single $1$ in  the position $j_\ell$, then $\Inv(S_\ell)=j_\ell-1$.  If $S_\ell$ is all copies of $1$, with a single $0$ in position $j_\ell$, then $\Inv(S_\ell)=c_\ell-j_\ell$, where $c_\ell$ is the length of $S_\ell$.
In either case, we denote by $u(j_\ell)$ the index $u$ of the partition in which the addable node represented by $j_\ell$ lies.

\begin{examp}\label{dif}
Let us take $a=3$, as in Figure 1 below, and consider the path $(0,1,0)$, this being a case where the value of $c_3$ depends on the value of $S_2$.  We have $c_1=3,  k^1=1$, and $\tilde S_1 =(1,0,0)$, so the $e$-regular multipartition after one step on the path  is $\mu_1=[(1),\emptyset,\emptyset;\emptyset, \emptyset,\emptyset]$ and that appears in the canonical basis with coefficient $1$.  If we choose a different choice sequence $S_1=
(0,0,1)$, with the $1$ in position $j_1$,  we get a multipartition  $[\emptyset,\emptyset,(1);\emptyset, \emptyset,\emptyset]$, which occurs in the canonical basis element $G(\mu_1)$ with coefficient $v^{\Inv(S_1)}=v^2$.

At the second stage, we have $c_2=2k^1+a=5$, $k^2=1$, and $\tilde S_2 =(1,0,0,0,0)$, giving an $e$-regular multipartition
$\mu_2=[(2),\emptyset,\emptyset;\emptyset, \emptyset,\emptyset]$. Now we pick a choice sequence $S_2=S_2^0 \vee S_2^1$, with $S_2^0$ of length $2$ and $S_2^1$ of length $3$. The choice sequence $S_2=(0,0,0,1,0)$, following $S_1$ above, would give a multipartition  $\lambda_2=[\emptyset,\emptyset,(1);\emptyset, (1),\emptyset]$, which would have a coefficient $v^{\Inv(S_1)+\Inv(S_2)}=v^5$.

Finally, at the third stage, $c_3$ depends on the structure of $S_2$. If $S_2^0=(0,0)$, so that the $1$-node is added to a $1$-corner partition, then $c_3$ is $4$,  there being two $0$-corner copies of $\emptyset$ which can be changed to $(1)$, and one $1$-corner partition $(1)$ with two addable $0$-nodes.  On the other hand, if $S^0_2=(1,0)$ or $(0,1)$, giving $U_1^1=(2)$ or $U_2^1=(1^2)$ respectively in position $j_1$, then $c_3=3$ and there is one addable $0$-node for each of the three $0$ corner partitions.

\end{examp}

In all the results below, we will have recourse continually to the Theorem 6.16 in Mathas \cite{M2}. The result there is stated for partitions rather than multipartitions, and Mathas is working with $e$-restricted rather than $e$-regular multipartitions.  To take care of these descrepancies, we will give a slightly different proof compatible with our set-up.  However,  when we use the result thereafter we will simple quote Mathas.

\begin{lem}\label{Mathas} If a multipartition $\mu$ has at least $\ell$ $i$-addable nodes, and if $\lambda$ is the result of adding $\ell$ $i$-addable nodes with choice sequence $S$ choosing among addable nodes, then $f^{(\ell)}_i(\mu)$ contains $\lambda$ with coefficient $v^{\Inv(S)}[\ell]_v!$.

\end{lem}
\begin{proof} For each permutation $\pi$ in the symmetric group $ S_\ell$,  define
\[
a(\pi) := \sum^\ell_{j=1}\#\{i:i<j,\pi(i)>\pi(j)\}-\#\{i:i<j,\pi(i)<\pi(j)\}
\]

This is closely related to the inversion munber of permutations, defined by 
\[
\Inv(\pi) := \#\{(i,j): i<j, \pi(i) > \pi(j)\}.
\]
\noindent Indeed, since there are $j-1$ natural numbers $i$ with $i<j$, clearly clearly
\[
\#\{i:i<j,\pi(i)>\pi(j)\}-\#\{i:i<j,\pi(i)<\pi(j)\}=2\#\{i:i<j,\pi(i)>\pi(j)\}-(j-1).
\]
\noindent and therefore 
\[
\alpha(\pi)=2\Inv(\pi)-{\ell\choose 2}.
\]

\noindent If the characteristic sequence $S$ contained only copies of $1$, we would nearly be finished, but since it may also contain copies of $0$, we also have to consider the inversion number $\Inv(S)$, from Definition \ref{Inv}.  We now claim that a copy of $\lambda$ obtained by adding the copies of $1$ in $S$ according to the permutation $\pi$ will have coefficient $v^{\Inv(S)}v^{a(\pi)}$. 

We let $S_0,S_1,\dots, S_\ell$ be the  characteristic sequences as we add copies of $1$ according to the partition $\pi$. Assuming the claim true for $\ell-1$ with permutation $\bar \pi$, we want to prove it for $\ell$ and permutation $\pi$. Assume that the last number $1$ we insert in $S_{\ell-1}$ is in the 
coordinate $t$, and that there are $s$ copies of $1$ before it in $S_{\ell-1}$.  By the induction hypothesis, the previous multipartition $\bar \lambda$ had coefficient $v^{\Inv(S_\ell)+a(\bar \pi)}$.  To add this new node, we must calculate the number of addable and removable $i$-nodes before position $t$, which is to say, the number of zeros minus the number of ones.  There are $t-1$ components to the vector before $t$, and of these, $s$ are ones, so we add $t-1-2s$ to the exponent of $v$.
Now $\Inv(S_\ell)-\Inv(S_{\ell-1})$ is $t-1-s$,the number of zeros in front of the new $1$, minus $\ell-1-s$, the number of copies of $1$ after $t$ which will be missing one zero, altogether $t-\ell$. The difference $a(\pi)-a(\bar \pi)$ will be $(\ell-1-s)-s$, since we add $1$ for each $i$ with $\pi(i)>s$ and subtract $1$ for each of the $i$ with $\pi(i) \leq s$.  Altogether,
\[
\Inv(S)+a(\pi)-\Inv(S_{\ell-1})-a(\bar \pi) =(t-\ell)+(\ell-1-2s)=t-1-2s
\]
\noindent as needed.
After adding the $\ell$ nodes which produce $\lambda$ in all possible orders, we thus get $\lambda$, multiplied by $v^{\Inv(S)}\sum_{\pi \in S_\ell} v^{a(\pi)}$.

It follows, from MacMahon's formula (see, for example, \cite{S}) for the inversion number generating functions that 
\[
\sum_{\pi \in S_{\ell}} v^{a(\pi)} = v^{-{ \ell\choose2}} \sum_{\pi \in S_\ell} v^{2\Inv(\pi)}=v^{-{ \ell\choose2}}\prod^{\ell}_{j=1}\frac{v^{2j}-1}
{v^2-1}=\prod^{\ell}_{j=1}\frac{v^{j}-v^{-j}}{v-v^{-1}}=\prod^{\ell}_{j=1}[j]_v=[\ell]_v!
\]
\end{proof}

Using the notation $t$ and $t'$ given in the conjecture, we   start with the case where $t=t'=1$, that is to say,  canonical basis elements at the top of the crystal,  reached by adding nodes of a single residue.   We now introduce the notation which will allow us to describe the multipartition occurring in the canonical basis of the multipartitions with defect $k(a-k)$.  We let $u$ with $1 \leq u \leq 2a=r$ be the range of superscripts indicating the various partitions in the multipartition, and we recall that $S_1^i$ is the subsequence of the characteristic sequence $S_1$  for which the addable nodes are contained in $i$-corner partitions. If the addable nodes are indexed by an integer $j$, we let $u(j)$ indicate the undex of the partition containing that addable node.  After the first step, there can be several addable nodes in a single partition.
\[
\tau^n_i(S_1)^u
\begin{cases}
T_{n+1},&(S_1^i)_{u-ia}=1, 1 \leq u-ia \leq a,\\
T_{n-1},&(S_1^i)_{u-ia}=0, 1 \leq u-ia \leq a,\\
T_n&  a+1 \leq u+ia \leq 2a,
\end{cases}
\]

\begin{lem}\label{d} For a dominant integral weight $\Lambda$ of an affine Lie algebra of type $A$, if the coefficient of $\Lambda_i$ in $\Lambda$ is $a$
then the $e$-regular multipartition  $f_i^{(k)} (u_\Lambda)$ for an integer $k$ is $\tau^0_i(\tilde S_1)$, with $0 \leq k \leq a$, and the canonical basis element is
\[
G(\tau^0_i(\tilde S_1))=\sum_{S_1 \in \mathcal S(a,k)}v^{\Inv(S_1)}\tau_i^0(S_1).
\]

 The shape of the canonical basis element
$G(\tau^0_i(\tilde S_1))$ is given by $(s(a,k,0),\dots, s(a,k,k(a-k))$ where $s(a,k,\ell)$ is a recursive function which is $0$ except for $0 \leq \ell \leq k(a-k)$, and satisfies the following recursion scheme:
\[
s(1,0,0)=s(1,1,0)=1,
\]
\[
s(a,k,\ell)=s(a-1,k-1,\ell)+s(a-1,k,\ell-k).
\]
\end{lem}

\begin{proof}
The $e$-regular multipartition $\mu=\tau^0_i(\tilde S_1)$ consists of $k$ partitions $(1)$ at the beginning of  the $i$-corner multipartitions, of which there are $a$. The only addable $i$-nodes in $u_\Lambda$ are the corners of those $a$ partitions, so by the formula for the actions of the divided power in the Fock space, there are ${a \choose k}$ different multipartitions which can occur in the canonical basis element $G(\mu)$ and they will all occur.  Let $S_1$ be  any such choice of $k$ partitions from $a$ copies of $0$-corner $\emptyset$, and let $b_1,b_2,\dots,b_k$ be the positions of the $1$-entries, on a scale from $0$ to $a-1$.  By Mathas \cite{M2} , 6.16, when we act on the highest weight vector $u_\Lambda$ by $f_i^k$, we get each $\tau^0(S_1)$ multiplied by $[k]_v!$. After dividing by the factorial, as we showed in the Lemma \ref{Mathas},  the power of $v$ which occurs as coefficient of $\tau^0(S_1)$ is $\Inv(S)$.

 The function $s(a,k,\ell)$ which gives the shape is then a function counting all the multipartitions with coefficient $v^\ell$.
For $a=1$, we have $0 \leq k \leq 1$, so $k=0,1$.  This means that $0 \leq \ell \leq 1(1-1)=0$, so $\ell=0$.   We have only the defect $0$ multipartitions and so get $s(1,1,0)=s(1,1,0)=1$.  Thereafter, the
number of multipartitions with coefficient $v^\ell$ is the sum of those starting with $1$, for which the power is determined by the remaining $a-1$ elements of the sequence, and those starting with $0$, for which the initial $\emptyset$ adds $k$ to the power of $v$  determined by the remainder of the sequence, giving the desired recursion formula.
\end{proof}

For every defect, there is a degree after which all weight spaces with this defect occur at the end of strings in the block-reduced crystal graph, and once this happens, almost all addable nodes have the same residue.  For the images of the multipartitions in the top rows under the action of the Weyl group, this is true from the very beginning.
\begin{cor*}\label{shape} The shape function in closed form:
\begin{itemize}
\item for $k=1$, $s(a,1,\ell)=1$,
 for $0 \leq \ell \leq a-1$.
\item for $k=2$, $s(a,2,\ell)= \lfloor{\frac{a-\mid \ell-(a-2) \mid}{2} }\rfloor$,$0 \leq \ell \leq 2(a-2)$.
\item for $k=3$, $s(a,3,\ell)=\sum_{t=1}^{\lfloor \frac{\ell}{3}+1 \rfloor}  \left \lfloor{\frac{a-t-\mid \ell-3t-(a-2) \mid}{2} }\right \rfloor$.
\end{itemize}
\end{cor*}
\begin{proof} Note that for $k=1,2$ the closed form is symmetric around $\frac{d}{2}$, where $d$ is the defect $k(a-k)$, and we presume this to be true in general. To prove that would probably require reformulating the recursion in terms of a symmetric parameter $\mid \ell - \frac{d}{2} \mid$.
\begin{itemize}
\item For $k=1$, the canonical basis element is svelte, since the multipartition multiplying $v^\ell$ will be that obtained from the highest weight vector by adding an $i$-node to partition $\ell+1$.
\item For $k=2$, we can separate $s(a,2,\ell)=s(a-1,1,\ell)+s(a-1,2,\ell-2)$.   If we continue to separate the term with $k=2$, each time adding $s(a-t,1,\ell-2(t-1))$ as long as  $\ell-2(t-1) \geq 0$, which means $t \leq \frac{\ell}{2}+1$, we get a sum of elements which are all $0$ or $1$, and we must count the number which are $1$:
\[
\sum_{t=1}^{\lfloor \frac{\ell}{2}+1 \rfloor} s(a-t,1,\ell-2(t-1)).
\]

 Case 1: If $\ell \leq \frac{d}{2}=a-2$, then $ s(a-t,1,\ell-2(t-1))=1$ from $t=1$ as long as  $\ell-2(t-1) \geq 0$, which is to say, until $t=\lfloor {\frac{\ell}{2}+1} \rfloor$, altogether  $\lfloor {\frac{\ell}{2}+1} \rfloor$ copies of $1$.  However, since   $\ell \leq a-2$,we have $|\ell-(a-2)|=a+2-\ell$ and thus
\[
\left  \lfloor{\frac{a-\mid \ell-(a-2) \mid}{2} }\right \rfloor= \left \lfloor{\frac{( \ell+2)}{2} } \right \rfloor,
\]
\noindent as desired.

 Case 2:  $\ell > a-2$, then for $t=1$ we have  $ s(a-t,1,\ell-2(t-1))=0$.  The first value of $t$ for which we get the value $1$ is when
$\ell-2(t-1) \leq a-t-1$, which is equivalent to  $\ell-a+3 \leq t$.  The total number of copies of $1$ in the sum is then $ \lfloor{\frac{( \ell+2)}{2} }\rfloor-(\ell-a+3)+1$.
This equals $ \lfloor{\frac{a-\mid \ell-(a-2) \mid}{2} }\rfloor$, as desired.

\item For $k=3$, by the same arguments we used above, the sum goes from $t=1$, which could give a value $0$, as long as $\ell-3(t-1) \geq 0$, giving
\[
s(a,3,\ell)=\sum_{t=1}^{\lfloor \frac{\ell}{3}+1 \rfloor } s(a-t,2,\ell-3(t-1)).
\]
Substituting from the result for case $k=2$, we get the desired formula.
\end{itemize}
\end{proof}
\begin{lem}\label{n}
A multipartition derived from the top rows under the action of a reduced word in the Weyl group generators of length $n$ is
\[
\tau^n_i(\tilde S_1).
\]

The canonical basis element of the images of a multipartition from the top rows under the action of a reduced word in the Weyl group generators of length $n$ is
\[
G(\tau^n_i(\tilde S_1))=\sum_{S_1 \in \mathcal S(a,k)}v^{\Inv(S_1)}\tau^n_i(S_1).
\]
The shape is given by the same recursive formula $s(a,k,\ell)$ given in the previous lemma.
\end{lem}

\begin{proof}

We are doing explicitly  the case residue $i=0$, the case $i=1$ being dual. We now apply  Theorem 6.16 from \cite{M2}  to each of the multipartitions $\tau^0(S_1)$ in $G(\tau^0(\tilde S_1))$.  The hub is $[a-2k,a+2k]$, so there are $a+2k$ addable $1$-nodes, two for each partition $(1)$ in the $0$-corner part, and $a$ for all the $1$-corner nodes. The result of adding all these nodes is exactly $\tau^1(S)$, and by the theorem quoted above, that is the result of acting on $\tau^0(S)$ by the divided power $f_i^{(2k+a)}$.  The result is a canonical basis element of exactly the same shape as before.

We now continue by induction, assuming that we have a canonical basis element of the desired shape, and calculating that the  number of addable nodes in
$\tau^n(S)$ must be $k(n+2)+(a-k)n +a(n+1)$, and the result of adding them all will be $\tau^{n+1}(S)$.  After applying the Mathas result \cite{M2}, 6.16 again, we get the desired canonical basis element. The case $i=1$ is dual.

\end{proof}

\ \noindent \begin{figure}[ht]
\centering
\includegraphics[bb=0 0 700 350]{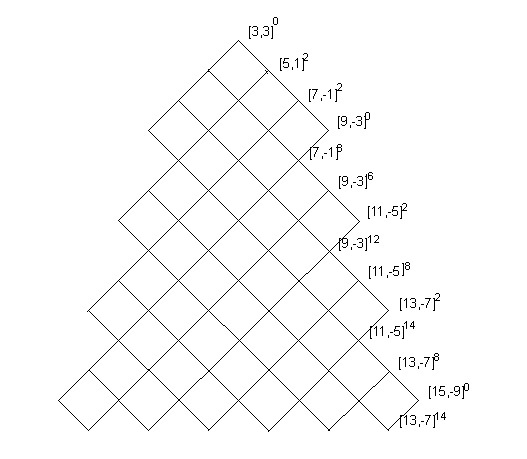}
\captionsetup{labelformat=empty}
\caption{Figure 1: $e=2,\Lambda=3\Lambda_0+3\Lambda_1$, truncated at degree $13$.}
\end{figure}

In Figure 1, we drew the symmetric  block reduced crystal for $a=3$, where the vertices on the right are labelled by the hub with the defect as superscript, and the vertices on the left are symmetric.  The label of any vertex in the interior can be obtained by going down the lattice, adding $r=6$ to the defect and leaving the hub the same.

\begin{lem}\label{n=0}
For a symmetric crystal for $\Lambda=a\Lambda_0+a\Lambda_1$ with $e=2$, any weight space which has content $(k,1)$ or $(1,k)$  for $1 \leq k  \leq a$  has dimension $2$ if $k=1$ and $3$ if $k \geq 2$. For each  path $p$ through the crystal, there is an integer $m$ such that:

\[
G(  \pi^0(\tilde S_1,\dots, \tilde S_{m})) =\sum_{S_1 \in \mathcal S(c_1,k^1)}\dots  \sum_{S_{t} \in \mathcal S(c_m,k^m)} v^{\Inv(S_1)+\dots+\Inv(S_m)} \pi^0(S_1,\dots,S_m)
\]
where
\begin{itemize}
\item $\mathbf{p=(0^k,1):}$  $m=2$, $t=2, c_1=a, k^1=k, S_1 \in \mathcal S(a,k), c_2 = a+2k, k^2=1, S_2 \in \mathcal S(2k+a,1)$ with the single $1$ in position $j_2$, and $\pi^0(S_1,S_2)$ is identical to $\tau^0(S_1)$, except for the following partitions:
\[
\pi^0(S_1, S_2)^u=
\begin{cases}
(2),& 1 \leq u \leq a,u=u(j_2), (S_2^0)_u=1,j_2 \equiv 1 \mod 2,\\
(1^2),& 1 \leq u \leq a,u=u(j_2), (S_2^0)_u=1,j_2 \equiv 0 \mod 2,\\
(1),& j_2>2k, u=u(j_2).
\end{cases}
\]
The case  $\mathbf{p=(1^k,0)}$ is dual.

\item $\mathbf{p=(1,0^k):}$   $m=2$, $t=2, c_1=a, k^1=1, S_1 \in \mathcal S(a,1)$ with the single $1$ in position $j_1$, $ c_2 = a+2,  k^2=k,  S_2 \in \mathcal S(a+2,k)$, and $\pi^0(S_1,S_2)$ is identical to $\tau^0(S_1)$, except for the following partitions:
\[
\pi^0(S_1, S_2)^u=
\begin{cases}
(1),&u \leq a, (S_2^0)_u=1,\\
(2),&  u=u(j_1), S_2^1=(1,0),\\
(1^2),& u=u(j_1), S_2^1=(0,1),\\
(T_2),& u=u(j_1), S_2^1=(1,1).
\end{cases}
\]
The case  $\mathbf{p=(0,1^k)}$ is dual.

\item $\mathbf{p=(0,1,0^{k-1})}$  $m=2$,  $t=3, c_1=a, k^1=1, S_1 \in \mathcal S(a,1)$ with the single $1$ in position $j_1$, $ c_2 = a+2,                k^2=1, S_2 \in \mathcal S(a+2,1)$ with a single $1$ in position $j_2$, $k^3=k-1$.   Then  $\pi^0(S_1, S_2, S_3)$ is identical to $\tau^0(S_1)$, except for a few special partitions depending on the values of $j_2$:

\noindent If $j_2 \leq 2, c_3=a, S_3=S_3^0 \in \mathcal S(a,k-1),$
\[
\pi^0(S_1, S_2,S_3)^u=
\begin{cases}
(1),&(S_3^0)_u=1, u \neq j_1,\\
(3),&(S_2^0)=(1,0), u=j_1, (S_3)_{j_1}=1,\\
(2),&(S_2^0)=(1,0), u=j_1, (S_3)_{j_1}=0,\\
(1^3),&(S_2^0)=(0,1),  u=j_1, (S_3)_{j_1}=1,\\
(1^2),&(S_2^0)=(0,1),  u=j_1, (S_3)_{j_1}=0.
\end{cases}
\]
\noindent If $j_2 >2 , c_3=a+1, S_3=\in \mathcal S(a+1,k-1)$,
\[
\pi^0(S_1, S_2,S_3)^u=
\begin{cases}
(1),&(S_3^0)_u=1, u<j_1\\
(1),&(S_3^0)_{u-1}=1,j_1< u \leq a\\
(2),&S_3^1=(1,0), u=u(j_2),\\
(1^2),&S_3^1=(0,1), u=u(j_2),\\
T_2,&S_3^1=(1,1),u=u(j_2).
\end{cases}
\]
The case  $\mathbf{p=(1,0,1^{k-1})}$ is dual.
\end{itemize}
\end{lem}
\begin{proof}
Before dividing into cases, we review the notation. If $j_\ell$ is the index of an addable node in a list of all addable nodes, then $u(j_\ell)$ is the index of the partition in which this node is located.

 All $e$-regular partitions for content $(k,1)$ are obtained by straightforward application of the signature method, with no removable nodes involved. The paths $(0^j,1,0^{k-j})$ give the same $e$-regular partition for any $1 \leq j  \leq k-1$, so we will assume $j=1$.
\renewcommand{\labelitemi}{$\circ$}
\begin{itemize}
\item Now we do the canonical basis elements, starting with path $p=(0^k,1)$ . This first case is easier, because it can be obtained by a single application of $f_1$ from the canonical basis element constructed in Lemma \ref{d}. The multipartition $\tau^0(S_1)$  has $k$ copies of $(1)$ in the $0$-corner partitions and every other partition $\emptyset$.  We now add the $j_2$th  $1$-node, and there are three possibilities.  If $j_2$ is odd and $j_2 \leq 2k$, then  we add to the side of the copy of $(1)$ in position $u=u(j_2)$  to get $(2)$ .  If  $j_2$ is even and $j_2 \leq 2k$, then  we add to the bottom of the copy of $(1)$ in position $u=u(j_2)$  to get $(1^2)$ .  If $j_2>2k$, then we add a new $1$-corner copy of $(1)$ in position $u$ with $u=u(j_2)=a+(j_2-2k)$. The multipartition $\tau^0(S_1)$  had a coefficient $\Inv(S_1)$, and now we multiply that by $v$ to the power ${j_2-1}=\Inv(S_2)$, the number of addable $1$-nodes above the one we just added.  There are no removable $1$-nodes because this is the first $1$-node that we are adding.

\item To calculate the case  $p=(1,0^k)$ we start with the canonical basis element of $f_1u_\Lambda$, which by Corollary \ref{shape} must be svelte, with the multipartitions dependent on a choice $j_1$ of a number between $1$ and $a$ for the position of the partition $(1)$ among the $0$-corner partitions. Each such multipartition is multiplied by $v^{j_1-1}$, because there are $j_1-1$ addable $1$-nodes before it. There are now $a+2$ addable $0$-nodes, and we let $S_2$ represent the choice of $k$ nodes from among them. Each of the resulting multipartitions is multiplied by an additional factor of $v^{\Inv(S_2)}$, as described in Lemma \ref{d}.

\item Finally, we have to calculate a longer path than the two before it, $(0,1,0^{k-1})$.  We start from a svelte canonical basis element  $G(\tau^0_0(\tilde S_1))$.  Each multipartition in the canonical basis has no removable $1$ nodes and $a+2$ addable $1$-nodes, these being the original $1$-corner partitions and two new addable nodes from the $0$-corner partition $(1)$, so we choose $j_2$  from $1$ to $a+2$ to give the position of the only $1$ in $S_2$. Altogether, the two choices multiply the resulting partition by $v^{\Inv(S_1)+\Inv(S_2)}$.  We now want to operate by $f^{(k-1)}$. There are $a-1$ $0$-corner nodes which have not been filled, and if the $1$  was added in a $1$-corner partition, it adds an additional $2$ addable $0$-nodes, giving $c_3=a+1$.  If the $1$-node was added to $(1)$ to give $(2)$ or $(1^2)$, then  there are $c_3=a-1+1=a$ possible $0$ nodes. Thus $c_3=a$ if $j_2=1,2$, and $c_3=a+1$ if $j_2>2$. There are thus two possible lengths for the choice functions $S_3$.

Case $j_2 \leq 2$:  In this case $S_2^0$ = $(1,0)$ or $(0,1)$ with $u=j_1$, we get either $(2)$  if  $j_2=1$ or or $(1^2)$ if $j_2=2$. Then $S_3=S_3^0$ and chooses $k-1$ among the $a$ possibilities. If the position chosen lies in a non-empty partition, we get $(3)$ or $(1^3)$, depending on the parity of $j_2$.

Case $j_2>2$:  In this case $S_2^1$ has a $1$ in position $j_2-2$, which means that $S_2$ adds a partition $(1)$ for $u=u(j_2)=a+j_2-2$. This means that $S_3^1$ has length $2$ and we get either $(2)$ or $(1^2)$ depending on the parity of $j_2$, while $S_3^0$ has length $a-1$ and adds $(1)$ partitions corresponding to the $1$-entries, can add them everywhere except $u-j_1$, so the $u$ above $j_1$ correspond to a position $u-1$ in $S_3^0$. Since $S_3^1$ has length $2$, it can be $(1,0),(0,1)$, or $(1,1)$, giving three possible partitions, $(2)$, $(1^2)$, or $T_2$.
\end{itemize}
\end{proof}

We are now ready to describe the remaining canonical basis elements with these defects. In order to formulate the theorem, we did computer calculation on many different cases.  Only when we were certain that we had identified all the different elements which were needed for each type of path did we formulate the algebraic proof.

\begin{prop}\label{general} All  external weight spaces with defect $(k-1)(a-k+1)+2a$ for $1 \leq k \leq a$ have  canonical basis elements  for $n \geq 1$ 
 depending on the path as follows:

\[
G(  \pi^{n}(\tilde S_1,\dots, \tilde S_{m}))=\sum_{S_1 \in \mathcal S(c_1,k^1)}\dots  \sum_{S_{t} \in \mathcal S(c_\ell,k^\ell)} v^{\Inv(S_1)+\dots+\Inv(S_m)} \pi^{n}(S_1\dots,S_m),
\]
\begin{itemize}
\item $\mathbf{p=(0^k,1^{2k+a-1},\dots):}$   $t=2, c_1=a, k^1=k, S_1 \in \mathcal S(a,k), c_2 = 2k+a, k^2=2k+a-1, S_2 \in \mathcal S(2k+a,2k+a-1)$ with the single $0$ in position $j_2$, and $\pi^n(S_1,S_2)$ is identical to $\tau^n(S_1)$ for $n \geq 1$, except for the following partitions:
\[
\pi^n(S_1, S_2)^u=
\begin{cases}
U_{1}^n,&  u \leq a, u=u(j_2), j_2 \equiv 0 \mod 2,\\
U_{2}^n,&  u \leq a,, u=u(j_2), j_2 \equiv 1 \mod 2,\\
T_{n-2},& u>a, u=u(j_2).
\end{cases}
\]
The case  $\mathbf{p=(1^k,0^{2k+a-1},\dots)}$ is dual.

\item $\mathbf{p=(1,0^k,1^{2k+a-2},\dots):}$   $t=2, c_1=a, k^1=1, S_1 \in \mathcal S(a,1)$ with the single $1$ in position $j_1$, $ c_2 = a+2, k^2=k, S_2 \in \mathcal S(a+2,k)$, and $\pi^n(S_1,S_2)$ is identical to $\tau^n(S_2^0)$ for $n \geq 1$, except for the following partitions:

If $S_2^1=(0,0), (1,0), (0,1)$, we need an additional characteristic sequence $S_3$ of length $2k+a-1$with a single $0$ in position $j_3$:
\[
\pi^n(S_1, S_2,S_3)^u=
\
\begin{cases}
T_{n+1},&u \leq a, u \neq u(j_3),( S_2^0)_u=1,\\
U^n_1,& u \leq a, u=u(j_3), j_3 \equiv 0 \mod 2, \\
U^n_2,& u \leq a,u=u(j_3), j_3 \equiv 1 \mod 2, \\
T_{n-2}& u>a,u \neq u(j_1), u=u(j_3),\\
U^n_1,& u > a, u=u(j_3), \\
U^{n+1},& u > a, u \neq u(j_3), S_2^1=(1,0),\\
U^n_2,& u > a,u=u(j_3), S_2^1=(1,0),\\
U^{n+1}_2,& u > a,u \neq u(j_3), S_2^1=(0,1),\\
U^{n}_2,& u > a,u = u(j_3), S_2^1=(0,1),\\
T_{n-2},&u>a,u=u(j_3), u \neq u(j_2).
\end{cases}
\]

If $S_2^1=(1,1)$, then there is no need for a third characteristic sequence. 

\[\pi^n(S_1, S_2)^u=
\begin{cases}
T_{n+1},&u \leq a, (S_2^0)_u=1,\\
T_{n+2}& u>a,u = u(j_1),\\
T_{n}& u>a,u \neq u(j_1).
\end{cases}
\]
The case  $\mathbf{p=(0,1^k, 0^{2k+a-2},\dots)}$ is dual.

\item $\mathbf{p=(0,1,0^{k-1},1^{2k+a-2},\dots)}$   $t=3,  c_1=a, k^1=1, S_1 \in \mathcal S(a,1)$ with the single $1$ in position $j_1$, $ c_2 = a+2, k^2=1, S_2 \in \mathcal S(a+2,1)$ with a single $1$ in position $j_2$.  The structure and length of $S_3$ depends on the value of $j_2$.  In all cases where it is needed, $S_4$ is of length $2k+a-1$, and has a single $0$ in position $j_4$. Then  $\pi^n(S_1,S_2,S_3,S_4)$ is identical to $\tau^n(S_1)$ for $n \geq 1$, except for a few special partitions depending on the values of $j_2,  S_3$  and $ S_4$:

\noindent If $j_2 \leq 2$,
then for $(S_3^0)_{j_1}=1$,
\[
\pi^n(S_1, S_2,S_3)^u=
\begin{cases}
T_{n+1},& u \neq j_1, (S^0_3)_u=1,\\
U_1^{n+1},&u=j_1,(S_2^0)=(1,0),\\
U_2^{n+1},&u=j_1,(S_2^0)=(0,1).
\end{cases}
\]
and for  $(S_3^0)_{j_1}=0$, we have
\[
\pi^n(S_1, S_2,S_3,S_4)^u=
\begin{cases}
T_{n+1},& u  \neq  j_1, (S^0_3)_u=1, u \neq u(j_4) ,\\
U_1^{n},& u \neq  j_1, (S^0_3)_u=1,u=u(j_4), j_4 \equiv 0+\lfloor \frac{u}{u(j_1)} \rfloor   \mod 2, \\
U_2^{n},& u \neq  j_1, (S^0_3)_u=1,u=u(j_4), j_4 \equiv 1 +\lfloor \frac{u}{u(j_1)} \rfloor  \mod 2, \\
T_{n+1},& u = j_1, u(j_4) \neq j_1, \\
U_1^{n},& u = j_1,  u(j_4)=j_1, S_2^0=(1,0), \\
U_2^{n},& u = j_1,  u(j_4)-j_1, S_2^0=(0,1), \\
T_{n-2},&u>a,u=u(j_4).
\end{cases}
\]
\noindent If $j_2 >2$,we set  $ \pi^n(j_1,S_2,j_3)^u$ equal to $\tau^n_0(S_2')^u$ except for the following cases:

if $S_3^1=(0,0)$,
\[
\pi^n(S_1, S_2,S_3)^u=
\begin{cases}
T_{n+1},&(u=j_1) \vee  (S_3^0)_u=1,u \neq u(j_4),\\.
U_1^{n},& u \leq a, u=u(j_4), j_4 \equiv 0 \mod 2,\\
U_2^{n},&  u \leq a,u=u(j_4), j_4 \equiv 1 \mod 2,\\
T_n, &u=u(j_2)>a,\\
T_n, &u>a, u \neq u(j_4),\\
T_{n-2}, &u>a, u = u(j_4).\\
\end{cases}
\]

if $S_3^1=(1,0),(0,1)$,
\[
\pi^n(S_1, S_2,S_3,S_4)^u=
\begin{cases}
T_{n+1}, &u=j_1 \vee  (S_3^0)_u=1,u \neq u(j_4),\\
T_{n-1},&u\neq j_1 \wedge  (S_3^0)_u=0,\\
U_1^{n},& u=u(j_4), j_4 \equiv 1 \mod 2,\\
U_2^{n},& u=u(j_4), j_4 \equiv 0 \mod 2,\\
U_1^{n+1},&u=u(j_2),S_3^1=(1,0),u \neq u(j_4),\\
U_1^{n},&u=u(j_2),S_3^1=(1,0),u=u(j_4),\\
U_2^{n+1},&u=u(j_2),S_3^1=(0,1),u \neq u(j_4),\\
U_2^n,&u=u(j_2), S_3^1=(0,1),u=u(j_4),\\
T_n, &u>a, u \neq u(j_4),u \neq u(j_2),\\
T_{n-2}, &u>a, u = u(j_4).
\end{cases}
\]

if $S_3^1=(1,1),k \geq 3$,
\[
\pi^n(S_1, S_2,S_3)^u=
\begin{cases}
T_{n+1},&u=j_1 \vee  (S_3^0)_u=1,\\
T_{n-1},&u\neq j_1 \vee  (S_3^0)_u=0,\\
T_{n+2}, &u=u(j_2),\\
T_n,&u>a,u \neq u(j_2).
\end{cases}
\]

The case  $\mathbf{p=(1,0,1^{k-1})}$ is dual.
\end{itemize}
\end{prop}
\begin{proof} We treat each case separately.
\begin{itemize}
 \item $\mathbf{p=(0^k,1^{2k+a-1},\dots):}$
In this case, we want to start the numbering at  $n=1$. The hub of $\tau^0_0(\tilde S_1)$ is $[a-2k,a+2k]$.  The $1$-string is of length $2k+a$ and ends at $\tau^1_0(\tilde S_1)$, whose canonical basis element is given in Lemma \ref{n}. Then the operation by $e_i$ on the canonical basis element give multipartitions which have been modified by removing the $j_2$th removable node.  If $1 \leq j_2 \leq 2k$, then if $j_2$ is odd, we replace one copy of the triangular partition $T_2$ by $(1^2)$ because we remove the first removable node, whereas if $j_2$ is even, we replace one $T_2$ by $(2)$, having removed the second $1$-removable node. On the other hand, if $j > 2k$, then we remove the $(j_2-2k)$th  of the $1$-corner partitions, leaving $\emptyset$ in that spot.  This gives exactly the values of $\tau^1_0(S_1)$, except for one partition, which is described in the proposition, its location depending on the value of $j_2$, so the case $n=1$ is solved for this path.

 The hub of this weight space is $[a+2k,a-2k]+(a+2k-1)[-2,2]=[-a-2k+2,3a+2k-2]$.  The $1$-string above this $1$-string is shorter by two vertices, so there is no point above this weight space in the $0$-direction.  Thus our weight space is external and lies at the beginning of a $0$-string of length  3a+2k-2.  In the original $\tau^1_0(\tilde S_1)$ there were exactly $3a+2k$ addable $0$-nodes, one for each of $a-k$ $0$-corner $\emptyset$, three for each $0$ corner $T_2$, and two for each $1$-corner $(1)$.  In each of the three different cases for $j_2$, the effect of removing the $1$-node was to reduce the number of possible $0$-nodes by $2$.  If we left out the first $1$-node in a $T_2$, then we can only add a single $0$-node at the bottom, giving $(1^3)$.  If we omitted the bottom $1$-node of a $T_2$, then we can only add a single $0$-node at the end, giving $(3)$. If we omitted a copy of $(1)$ among the $1$-corner nodes, then we cannot add any $0$-corner nodes at the spot. Altogether, as we go down the $0$-string, we add all the $0$-nodes to every multipartition, so there are no choices, and the shape of the multipartition at the end of the string is exactly as it was at the beginning, and we get the formula in the proposition for $n=2$ and all the multipartitions appearing in the canonical basis have only $1$-addable nodes. We now continue by induction, since for each $n$, all the partitions occurring in the formulae for the multipartitions for have only addable nodes, in this case,  of the parity opposite to that of $n$.  Since the vertex is external, the number of addable nodes equals the length of the string, so adding all addable nodes gets one to the vertex of the same defect at the other end of the string. There are never any choices and the shape of the canonical basis element is preserved.

The dual case $(1^k,0^{2k+a-1},\dots)$ is very similar.

\item  $\mathbf{p=(1,0^k,1^{2k+a-2},\dots):}$
Here $1 \leq j_1 \leq a$ chooses one of the $1$-corner partitions, in position $u(j_1)=a+j_1$.  Then $S_2$ distributes $k$ $0$-nodes, of which none, one, or  two can be places on a $1$-corner partition, so that the total number of possibilities is ${a+2 \choose k}$.  We now need to determine the number of addable $1$-nodes
and identify the multipartitions. In the case $n=1$, which is after adding one $1$-node, $k$ $0$-nodes, and $2k+a-2$ $1$-nodes, most of the $0$-corner partitions are $T_2$ or $\emptyset$, and most of the $1$-corner partitions are $(1)$.  The various special cases depend on $S_2^1$ and on $j_3$, and we will go over them now for the case $n=1$, letting $k'$ be the number of $1$-entries in $S_2^0$:

\begin{itemize}
\item $S_2^1=(0,0)$: If all of the $k$ $0$-nodes are in the $0$-corner section, then  there are $2k$ addable $1$-nodes in the $0$-corner section, and $a-1$ addable $1$-nodes in the $1$ corner section, as we replace $\emptyset$ with $(1)$ giving $2k+a-1$ altogether so a choice secquence $S_3$ is necessary. Thus in applying $f_1^{2k+a-2}$ to the canonical basis element, the characteristic sequence $S_3$ is all copies of $1$, except for a $0$ in position $j_3$ where
$1 \leq j_3 \leq 2k+a-1$.
  In this case $k'=k$.  If  $j_3 \leq 2k$, then there are $k-1$ copies of $T_2$, and one copy of $1^2$ or  $(2)$ depending on the parity of $j_3$.

\item $S_2^1=(1,0),(0,1)$. If there are $k-1$ $0$-nodes in the $0$-corner section, this gives $2(k-1)$ addable $1$-nodes.  There are still $a-1$ copies of $\emptyset$ to be filled, but in addition we now have either $(2)$ or $(1^2)$ at the previously chosen $1$-corner partition, and this can be converted to $(3)$ or $(1^3)$ respectively, giving an additional addable $1$ node, so that we have $2k+a-1$ altogether, as before, and again a characteristic sequence is necessary.  If $j_3 > 2k'$ but corresponds to not filling one of the partitions $\emptyset$, then the partition in the $1$-corner partition numbered by $j_1$ is $(1)$, $(2)$, $1^2$, or $T_2$ after the $0$-nodes are filled in, and becomes $(1)$, $(3)$, $(1^3)$, or $T_3$ after all the $1$ nodes are filled in.

\item $S_2^1=(1,1)$ Finally, we come to the case where $k \geq 2$ and there are only $k-2$ $0$-nodes in the $0$-corner partitions, giving $2(k-2)$ addable $1$-nodes there. As before we have the $a-1$ copies of $\emptyset$ to be filled, but now there is also a $1$-corner copy of $T_2$, to which $3$ different $1$ nodes can be added, altogether $2k+a-2$ nodes, the total number we need to add, so here we do not need a characteristic sequence $S_4$, and we continue the results in this case from Lemma \ref{n=0}
\end{itemize}
As before, the induction results from the weight space being external and  from noting that all the partitions have only addable nodes and no removable nodes.

\item  $\mathbf{p=(0,1,0^{k-1},1^{2k+a-2},\dots)}$: We have $u(j_1)=j_1$. 

If $j_2 \leq 2$, then we are adding a $1$-entry to the  $0$-corner $(1)$ in position $u(j_1)$. If $S^0_2=(1,0)$, then we get $(2)$, and if $S^0_2=(0,1)$, then we get $(1^2)$.  Next we need to add $k-1$ $0$-nodes, which must all go into $0$-corner partitions, so that $S_3=S_3^0$ with $k-1$ entries equal to $1$  and is of length $c_3=a$. We now need to distinguish two cases:

$(S^0_3)_{j_1}=1$:
In this case, adding a $0$-node to $(2)$ gives $(3)$, or adding a $0$-node to $(1^2)$ gives $(1^3)$.  Both of these have $2$ addable $1$-nodes.  In addition, we have added $k-2$ $0$-nodes to copies of  $\emptyset$, giving an additional $2(k-2)$ addable $1$-nodes. Putting these together with the $a$ addable $1$-nodes in the $1$-corner partitions, we have $2k+a-2$ addable $1$-nodes, just the number we need to add, so there is no need for a characteristic sequence $S_4$.

$(S^0_3)_{j_1}=0$: In this case, we added $k-1$ $0$-nodes in place of $\emptyset$, each giving $2$ addable $1$-nodes, but we also have another addable $1$-node in the partitions in position $u=u(j_1)$, giving, together with the $a$ addable $1$-nodes in the $1$-corner partitions, a total of $2k+a-1$ addable one nodes, which is one too many.  Therefore, we need a characteristic sequence $S_4$ of length $c_4=2k+a-1$, which will choose $2k+a-2$ addable nodes.  Letting $j_4$ be the position of the single $0$-entry in $S_4$, we consider the effect of adding all addable $1$-nodes to the existing partition, and combine that with considering the  various possible positions in which $j_4$ can lie.  In the $0$-corner partitions, if we add two $1$-nodes to a $(1)$, we get $T_2$, while if we omit one of them, we get $(2)$ or $(1^2)$.  If $u<u(j_1)$, then  $j_4\equiv 0 \mod 2$ means that we omit the bottom node, giving $(2)$ and if    $j_4 \equiv 1 \mod 2 $ we omit the side node, giving $(1^2)$.  If $j_4$ corresponds to $u=j_1$, then we remain with what we had, dependant on $S_2^0$.  If $u>u(j_1)$, then we again get  $(2)$ or $(1^2)$ , but we have to shift down by one, because there was only one possibility for $u(j_4)=u(j_1)=j_1$, so the required  parities of $j_4$ are reversed.  In order to compensate for this shift, we added $\lfloor \frac{u}{u(j_1)} \rfloor $, which is equal to $1$ when $u>j_1$.
If $u>a$, then all the partitions are $T_{n}$ as in $\tau^n(S_1)$, and if $j_4$ is in this section, then for $u=u(j_4)$ we get $T_{n-2}$.

If $j_2>2$, there are three cases, depending on the value of $S_3^1$, which is of length $2$.

If $S_3^1=(0,0)$, then all the $0$-nodes are in the $0$-corner partitions.  Each has two addable nodes, and when we add the $1$-nodes, we get $T_2$ if $u=j_1$ or $(S_3^0)_u=1$.
Together with $a-1$ addable $1$-corner $1$-nodes, we get $2k+a-1$, so we need a characteristic sequence $S_4$.  When  $u=j_1$ or $(S_3^0)_u=1$ but $u=u(j_4)$, then we get    $(2)$ or $(1^2)$, depending on the parity of $j_4$. As for the $1$-corner nodes, they are all filled with $(1)$ except possibily when $u(j_4)>a$, in which case we get $\emptyset$.  Since the weight space is external, we continue to $n>1$ by filling all nodes.

If $S_3^1=(1,0)$ or $(0,1)$, we again need $S_4$.  The possibilities for $0$-corner partitions are the same as in the previous case, but now for $u>a$ and $u=u(j_2)$ we have a few new possibilites.  Adding the $0$-nodes gives $(2)$ or $(1^2)$.  Then in the continuation, this becomes $(3)$ or $(1^3)$ if we don't have $u=u(j_4)$ and stays as it was if $u=u(j_4)$.

If $S_3^1=(1,1)$, then there is no need for $S_4$ and no partitions which are not equal to their own transpose.  For $n=1$ we get $T_3$, and for larger $n$ we get $T_{n+2}$ at that spot.

\end{itemize}

Finally, the induction.  The action by $e_0$ gives $0$, so going down the string from one end to the other involves adding all the addable nodes with no choices.  This sends $T_m$ to $T_{m+1}$, sends $U_i^n$ to $U_i^{n+1}$ for $i=1,2$.
\end{proof}

\begin{cor*}
For every multipartition $\mu$ occuring in a canonical basis element as in the Proposition \ref{general},  the multipartition $\mu^T$ in which every partition is transposed also occurs.

\end{cor*}

\begin{proof}
The partitions $T_n$ are all  transpose to themselves, and the only other partitions which occur are the $U_i^n, i=1,2$, which always occur in pairs, so that if one exists in the canonical basis element, the other occurs in the same position.

\end{proof}
We give one more lemma of a different flavor, being concerned not with defects showing up near the top of the block reduced crystal, but with defects appearing above a defect $0$ block in an $i$-string.

\begin{lem}\label{svelte} For $e=2$, and a symmetric crystal, consider an $i$-string of length $s$ ending with a block with weight $\lambda$ of defect $0$. The canonical basis element of the multipartition one up the string is svelte.
\end{lem}

\begin{proof} We begin with a defect $0$ multipartition which is obtained from the highest weight vector by action of the Weyl group.

A symmetric crystal, in addition to being symmetric, also looks rather similar to a spruce tree, with lower branches much longer that upper branches.  Then we start on the right side and act by the elements of the Weyl group, the hubs of vertices of defect $0$ are $[3a,-a], [-3a,5a],[ 7a,-5a], \dots$ with the coordinates reversed on the left side. At the top, the hub before $[3a,-a]$ is $[3a-2,2+a]$  Furthermore, since the hubs going down always drop by two in length and start at the same place, the vertex $\mu$ one above the defect $0$ vertex is also the beginning of an $i'$-string.   We are trying to show that the canonical basis element of $\mu$ is svelte.  The only multipartitions which can occur in $G(\mu)$ are those which can produce the multipartition $\lambda$, whose removable nodes are all of residue $i$, so it must be obtained by removing one removable node of residue $i$.

The length of the string was $(2c+1)a$, so there are $(2c+1)a$ nodes which were added and can be removed, giving $(2c+1)a$ candidate multipartitions.  Since, by an argument similar to that in Lemma~ \ref{k}, the defect of $\mu$ is $(2c+1)a-1$, we have to show that they all occur.  Now, in order to get the correct
canonical basis element for $\lambda$, we simply apply $e_i$ to $\lambda$.  The formula for $e_i$ has us remove the removable nodes one-by-one, multiplying each time by $v$ to a power which is the number of removable nodes below the one we are removing, minus the number of addable nodes.  Since there are no addable nodes, and the number of removable nodes is the same as the length of the string, we get a different power of $v$ for each of $(2c+1)a$ removable nodes, running from $0$ for the bottom node to $(2c+1)a-1$ for the top removable node.  This gives a svelte canonical basis element.

\end{proof}

\section{BLOCKS OF SMALL DEFECT}

The defect $0$ case is as already described, even in the non-symmetric case, so we will begin with defect $1$.  There can be a block of defect $1$ only if it occurs in the first string going out from the highest weight element, by results in \cite {BFS}. Furthermore, since the defects rise   towards the center of the string in a parabolic fashion as described in that paper, the only possible values of $a$ for which the crystal can contain a block of defect $1$ are $a=2$, for which the defects in the highest string are $0-1-0$.  In this case, we can simply determine all blocks of defect $1$ by using the action of the Weyl group. The sequence of multipartitions  with path beginning at zero consists entirely of triangular partitions, the first of side $n$, the second of side $n-2$, and the last two of sides $n-1$.

From the structure of the block reduced crystal graph, the defects which can occur in a symmetric crystal are the defects appearing in this first row, modulo $2a$.  Thus
\begin{itemize}
\item for $a=1$, the defects are all even numbers,
\item for $a=2$, the defects are congruent to $0$ or $1$, modulo $4$,
\item for $a=3$, the defects are congruent to $0$ or $2$, modulo $6$,
\item for $a=4$, the defects are congruent to $0$ , $3$ or $4$, modulo $8$.
\end{itemize}

We now turn to the case of defect $2$.   This can occur only when $a=1$  where it lies on a string with defects $0-2-2-0$ and the block of defect $2$ in internal, or for  $a=3$.  When $a>2$,  there must be a $k$ with $k(a-k)=2$, and this happens only when $a=3$ and $k=1,2$.  The multipartitions with defect $2$ have a very distinctive form, and we can calculate all of them.
The first example in defect $2$ is in degree $2$ and and there are two $e$ regular multipartitions, both svelte. In degree $3$, there are two blocks of defect $2$, each of which has one canonical basis element which is svelte,  corresponding to the path $(1,0,0)$ or $(0, 1,1)$ as in Lemma  \ref{svelte} and two which are not, corresponding to the alternating paths $(0,1,0)$ and $(1,0,1)$. This gives an example to show that the action of the Weyl group on internal vertices of a string need not preserve the shape of the canonical basis element.

\begin{itemize}
\item \textit{$a=1$} In this case the $e$-regular multipartitions are in one of two dual forms:
\begin{enumerate}
\item  For $\mu=[ T_{n+1}, T_{n-2}],$

 or $\mu^\diamond=[T_{n},U_1^n=(n+1) \vee T_{n-2}],$
\item for $\mu = [U_1^{n}=(n+1) \vee T_{n-2}, T_{n-2}],$

or  $\mu^\diamond=[U_1^n=(n+1) \vee T_{n-2}, T_{n}].$
 \end{enumerate}
In the first case, there are three monomials in all the $G(\mu)$. so they are all svelte.  The coefficient for $v$ in $\mu$ was given by taking the transpose of the non-triangular partition.

 In the second case, $G(\mu)$ is not svelte.  There are two multipartitions  multiplied by $v$, being given by the transposes of  $\mu$ and $(\mu^\diamond)'$ as in  the corollary to Proposition \ref{general}. For example, in degree $3$, we have
\[
G([(3),\emptyset])=[(3),\emptyset] +v[(1^3),\emptyset]+v [(1),(2)] +v^2 [(1),(1^2)].
\]
 This can be checked easily for the cases of lowest degree.  Thereafter, we appeal to Prop. \ref{general}, and note that in the form given, adding all addable nodes preserves the property of being transpose. 

\item \textit{$a=3$}
\begin{enumerate}
\item  For $\mu=[ T_{n+1},  T_{n-1},  T_{n-1}, T_{n}, T_{n}, T_{n}]$,

or $\mu^\diamond=[ T_{n}, T_{n}, T_{n}, T_{n+1},  T_{n-1},  T_{n-1}]$.
\item for $\mu=[ T_{n+1},  T_{n+1},  T_{n-1}, T_{n}, T_{n}, T_{n}]$,

or $\mu^\diamond=[ T_{n}, T_{n}, T_{n},  T_{n+1},  T_{n+1},  T_{n-1} ]$.
\end{enumerate}
as follows from Lemma \ref{d}.
\[
[(1), \emptyset,\emptyset, \emptyset, \emptyset,\emptyset]+v[ \emptyset,(1),\emptyset, \emptyset, \emptyset,\emptyset]+v^2[ \emptyset,\emptyset,(1), \emptyset, \emptyset,\emptyset].
\]

To produce the middle terms, we move the larger triangle down, and this is preserved under adding all addable nodes, so by Lemma \ref{d} we get the desired structure of all the canonical basis elements.
\end{itemize}

\end{document}